\documentclass[12pt, draft]{article}
\usepackage{amsmath, amsthm, amsfonts,amssymb}

\theoremstyle{plain}
\newtheorem{Theorem}{Theorem}
\newtheorem{Corollary}{Corollary}
\newtheorem{Lemma}{Lemma}

\newcommand	\Pro[1][\Q]		{{\mathbb P}_1(#1)}
\newcommand	\Q		         {\mathbb Q}

\newcommand	{\SL}[1][\Z]	{\sym{SL}(2,#1)}
\newcommand	\Z		         {\mathbb Z}

\newcommand	\cent		      {{\cal C}}
\newcommand	\dime		      {{\dim \jac}}
\newcommand	\e[1]		      {{e^{2\pi i\left({#1}\right)}}}
\newcommand	\gam		      {{\Gamma(4m)}}
\newcommand	\gamone		   {{\Gamma(1)}}
\newcommand	\half		      {{\frac 12}}
\newcommand	\jac		      {S_{k,m}(\Gamma)}
\newcommand	\mat[4]		   {(#1,#2;#3,#4)}
\newcommand	{\skewJ}[3]    {J^+_{#1,#2}(#3)}
\newcommand \sym[1]        {\operatorname{#1}}

\title{%
Memorandum on
Dimension Formulas for Spaces of Jacobi Forms}

\author{%
Nils-Peter Skoruppa}

\date{%
}

\begin{document}
\maketitle
\begin{abstract}
We state ready to compute dimension formulas for the spaces of Jacobi
cusp forms of integral weight $k$ and integral scalar index
$m$ on  subgroups of $\SL$.
\end{abstract}

\section{Introduction}

Denote by $\jac$ the space of Jacobi cusp forms of weight
$k$, index $m$ on a subgroup $\Gamma$ of finite index in $\gamone=\SL$.
In \cite{S-Z1} one finds an explicit trace formula
for Jacobi forms. One of the first applications of such a trace formula is to
calculate the dimensions of the spaces of Jacobi forms.
Although the cited trace
formula is a ``ready to compute'' formula,
it can still be considerably simplified if
one is merely interested in dimensions, i.e.~the trace  of the identity operator.
That this can be done, what one has to do and what the
out-coming formula is looking like, at least qualitatively, is without doubt
known to specialists. Nevertheless, there is no place in the literature where
this has been put down in sufficient generality.
The purpose of the present note
is to fill this gap.
The resulting dimension formulas are summarized in
Theorems~\ref{torsionsfree-gamma-naive} to~\ref{gamma-without-minus-one}.

\section{A first computation}
\label{computation}

We start with the trace formula  as
given in \cite[Theorem 1]{S-Z1}. According to this theorem one has
$$
\dime=\sum_A
I(A)g(A)+\sum_B I(B)g(B).
$$
Here the notation is as follows:
The symbol $\Gamma$ denotes an arbitrary  subgroup of (finite index in)
$\gamone$. In the first sum $A$ runs through a complete
set of representatives for the $\Gamma$-conjugacy classes of all non-parabolic
elements of~$\Gamma$, and in the second sum $B$ runs through a complete set of
representatives of all parabolic elements of $\Gamma$ modulo the equivalence
$\sim$, where $B_1\sim B_2$ if and only if $GB_1$ is $\Gamma$-conjugate to
$B_2$ for some $G\in \cent_{\Gamma\cap\gam}(B_1)$. Here, for any given
matrix $B$, parabolic or not, and any given subgroup
$\Gamma$ of $\gamone$, the symbol $\cent_\Gamma(B)$ stands for the
centralizer of $B$ in $\Gamma$. Moreover, $$I(1)=\left\lbrack \gamone:\Gamma
\right\rbrack\cdot {\frac {2k-3}  {48}},$$ and for parabolic B with positive trace
$$
I(B)=-\half\left\lbrack\cent_\Gamma(B):\cent_{\Gamma\cap\gam}(B)\right\rbrack^{-1}
\cdot\left(1-i\,C\left({\frac r s}\right)\right),
$$
where $r,s$ stand for those
uniquely determined positive integers such that $B$ and
$\cent_{\Gamma\cap\gam}(B)$ are $\gamone$-conjugate to\footnote{%
We use $\mat {a}{b}{c}{d}$ to denote matrices
$\begin{pmatrix}a&b\\c&d\end{pmatrix}$.
}
$\mat {1}{r}{0}{1}$ and
$\langle\mat {1}{s}{0}{1}\rangle$, respectively, and where $C(z)=\cot
(\pi z)$ for $z\not\in\Z,$ and $C(z)=0$ for $z\in\Z.$
For all other~$A$, the expression $I(A)$ is somehow
defined and will be recalled later;
the only important point for the moment is that $I(A)=0$ for non-split
hyperbolic $A$ (i.e. for those $A$ with trace t satisfying
$t^2-4\not=$ square in~$\Q^*$). Finally, $g(1)=2m$, and for a parabolic $B$
which is $\gamone$-conjugate to $\mat {1}{r}{0}{1}$ for some $r$ one has
$$
g(A)=\sum_{\lambda\bmod 2m}\e{\frac {r\lambda^2}  {4m}}.
$$

If
the reader wishes to compare the above formula for the dimensions with the
formula given in \cite[Theorem 1]{S-Z1} he should note that (i)
$\dime =\sym{tr}(H_{k,m,\Gamma}(\Gamma\ltimes\Z^2),S_{k,m}(\Gamma))$ in the
notation of \cite{S-Z1}; (ii)~we have dropped here various subscripts
and parameters: in the notations of \cite{S-Z1} we have
$I(A)=I_{k,m,\Gamma}(A)$, and $g(A)=g_m(\Gamma\ltimes\Z^2,A)$;
(iii)~$g_m(\Gamma\ltimes\Z^2,A)=G_m(A)$,
where the latter expression is given by \cite[Theorem 2]{S-Z1}
(when applying this theorem, note that the quadratic forms
$Q_A$ and $Q_A^\prime$ occurring
in the statement of the theorem are
equivalent modulo $\gamone$ if $A$ and $A^\prime$ are
$\gamone$-conjugate; this implies that $G_m(A)$ depends
only on the $\gamone$-conjugacy class of $A$; in particular $g(B)=g(\mat
{1}{r}{0}{1})$ for the parabolic $B$ as above, and
$Q_{\mat {1}{r}{0}{1}}(\lambda,\mu)=r\lambda^2$).

To be correct it must be added
that the quoted dimension formula holds strictly true only for $k\geq 3$. The given
formula becomes true for arbitrary~$k$ if one subtracts on the left hand side a certain correction term, which
for $k=1,2$, however, turns out to be non-trivial (cf.~\cite[formulas (9), (10) of \S 3]{S-Z1}). In fact, it can be shown that this correction term equals $\skewJ{3-k}m\Gamma$,
where  $\skewJ{3-k}m\Gamma$ denotes the space of
skew-holomorphic Jacobi forms of weight $3-k$, index $m$ on $\Gamma$
(as defined e.g.~in~\cite{S2}).
Thus the given formula holds true for arbitrary $k$ if one
replaces the left hand side by $\dime -
\dim \skewJ{3-k}{m}\Gamma$,
keeping in mind that $\dim \skewJ{3-k}{m}\Gamma=0$  for~$k\ge 3$.
The dimensions of
$\skewJ1k\Gamma$ and $\skewJ2k\Gamma$ can be explicitly calculated for
congruence subgroups $\Gamma$ using the Theorem of Serre and Starck on modular forms of weight 1/2 (cf.~\cite{S1} or~\cite{I-S} for details of the method which has to be applied). Thus, in principle, it would be
possible to give an effective formula for $\dime$ for arbitrary~$k$. However,
for simplicity we concentrate here on the case $k\geq 3$ and leave the
correction terms undetermined for $k\le 2$.

Furthermore, we assume first of all that $\Gamma$ contains nor elliptic matrices neither
the matrices with trace $-2$, i.e. that $\Gamma$ is torsion-free and that the cusps of
$\Gamma$ are all regular, i.e. that any parabolic subgroup of $\Gamma$ is
$\gamone$-conjugate to $\langle\mat {1}{b}{0}{1}\rangle$ for a suitable $b$.
Note that all these assumptions hold for the principal congruence subgroups $\Gamma(N)$ with $N\geq 3$
(as it follows easily from the fact that any elliptic matrix in $\gamone$ is
conjugate to one of the matrices $\pm\mat 0{-1}10$, $\pm\mat0{-1}11$ ore $\pm\mat{-1}{-1}10$).
Under these assumptions only $A=1$ and parabolic $B$
with trace$=2$ contribute to the given formula for $\dime$ (Here one has also to
use that $\gamone$ contains no split hyperbolic matrices, i.e. matrices with
trace$^2-4=$ square in $\Q^*$).

Concerning the parabolic contribution one easily verifies the following:

(i) $\cent_\Gamma(A)=\Gamma_p$ (=stabilizer of $p$ in $\Gamma$) for all
parabolic $A\in\Gamma$ with fixed point $p\in\Pro$;

(ii) for any two parabolic $A$ and $A^\prime$ there exists a matrix
$G\in\cent_\Gamma(A)$ such that $GA$ and $A^\prime$ are
$\gamone$-conjugate if and only if the fixed points $p$ and $p^\prime$ of $A$ and
$A^\prime$ are equivalent modulo $\Gamma$;

(iii) for any two parabolic $A$ and $A^\prime$ having the same fixed point one
has $A\sim A^\prime$ if and only if $A$ and $A^\prime$ lie in the same coset
modulo $\cent_{\Gamma\cap\gam}(A)$.

Taking into account these facts the parabolic contribution can now be written as
$\sum_{p\in\Pro}t_p$ where
$$
t_p=\sum_{A\in\Gamma_p/(\Gamma_p\cap\gam)}I(A)g(A).
$$
To simplify the $t_p$ fix a cusp $p$. Then there exist uniquely
positive integers  $b,f$ such that $\Gamma_p$ and $\Gamma_p\cap\gam$ are
$\gamone$-conjugate to $\langle \mat {1}{b}{0}{1}\rangle$ and $\langle\mat
{1}{bf}{0}{1}\rangle$, respectively. Thus
$$
t_p=\sum_{0<\nu\leq f}I(R\mat
{1}{b\nu}{0}{1}R^{-1})g(R\mat
{1}{b\nu}{0}{1}R^{-1})
$$
with a suitable $R\in\gamone$. Inserting the
quoted values for the functions $I$ and~$g$ one obtains
$$
t_p=-{\frac 1
{2f}}\sum_{0<\nu\leq f}\left(1-i\,C\left({\frac \nu f}\right)\right)
\sum_{\lambda\bmod 2m}\e{\frac {b\nu\lambda^2}  {4m}}.
$$
Now
$f={\frac {4m}  {(4m,b)}}$ (note that $\Gamma_p\cap\gam$ is $\gamone$-conjugate to
$\langle \mat {1}{{[b,4m]}}{0}{1}\rangle$ on the one hand, and to $\langle \mat
{1}{bf}{0}{1}\rangle$, by the definition of $f$, on the other hand; thus $bf=[b,4m]$,
whence $f={\frac {4m}  {(4m,b)}}$). Using this we can write
\begin{equation*}
\begin{split}
t_p=
&-\half\sharp\lbrace\lambda\bmod 2m
\vert b\lambda^2\equiv 0\bmod 4m\rbrace\\
&+{\frac m {f^2}}i\sum_{\nu\bmod f}
C\left({\frac \nu f}\right)
\sum_{\lambda\bmod f}
\e{{\frac {{\frac b {(4m,b)}}\nu\lambda^2} f}}.
\end{split}
\end{equation*}
The first term equals
$-{\frac m f}Q(f)$ where $Q(n)$, for any positive integer $n$, denotes the
greatest integer whose square divides $n$.

To simplify the second term we apply
the following Lemma, which is
Proposition A.2 in \cite{S-Z2}; for the proof the reader
is referred to loc.~cit.. 
\begin{Lemma}
Let $a$ and $f$ be positive integers. Then
$$
{\frac i f}\sum_{\nu\bmod f}C\left({\frac \nu f}\right)\sum_{\lambda\bmod f}
\e{\frac {a\nu\lambda^2}  f}
=-2(a,f)\sum_\Delta\left({\frac \Delta {a/(a,f)}}\right)H(\Delta).
$$
Here the sum on the right side is over all $\Delta<0$ dividing ${\frac f {(a,f)}}$
such that ${\frac f {(a,f)\Delta}}$ is square-free, and $H(\Delta)$ denotes the
Hurwitz class number of $\Delta.$
\end{Lemma}

Recall that $H(\Delta)$ equals the number of
$\gamone$-equivalence classes of all integral, positive definite
binary quadratic
forms of discriminant $\Delta$, counting forms $\gamone$-equivalent to a multiple
of $x^2+y^2$ (resp.~$x^2+xy+y^2$) with multiplicity~$\half$ (resp.~$\frac 13$).
In particular,
$H(\Delta)=0$ if $\Delta\not\equiv 0,1\bmod 4$.

Accordingly to this Lemma the second term of the last formula for $t_p$ can be
written as
$-{\frac {2m}  f}\sum\left({\frac \Delta {b/(4m,b)}}\right)H(\Delta)$ with $\Delta$
running through all negative integers dividing $f$ such that $f/\Delta$ is
square-free.

\section{A special case} 
\label{a-special-case}

Summing up the result of the calculations of the foregoing section,
we have proved
\begin{Theorem}
\label{torsionsfree-gamma-naive}
Let $k$ and $m$ be integers, $m\ge 1$. Let $\Gamma$ be a torsion-free
subgroup of finite index in $\gamone$ which contains no matrices with trace
equal to $-2$. Then the dimension of the space of Jacobi cusp forms of weight $k$,
index $m$ on $\Gamma$ is given by
\begin{multline*}
\dime - \dim\skewJ{3-k}m\Gamma=
m\cdot[\gamone:\Gamma]\,{\frac {2k-3}  {24}}\\
-\sum_p{\frac m {f_p}}Q(f_p)
-\sum_p{\frac {2m} {f_p}}
\sum_{\begin{subarray}{c}
\Delta\vert f_p,\;\Delta<0\\
f_p/\Delta \text{ squarefree}
\end{subarray}}
\left({\frac \Delta { {b_p} / {(4m,b_p)}}}\right)H(\Delta).
\end{multline*}
Here $p$ runs through a set of representatives for $\Gamma\backslash\Pro$, and
for each such~$p$ we use $b_p=\half [\gamone_p:\Gamma_p]$ ,
$f_p={ {4m} / {(4m,b_p)}}$. Moreover, $H(\Delta)$ denotes the Hurwitz class
number (as explained in the last but not least paragraph of section~\ref{computation}), and $Q(n)$, for any
positive integer $n$, denotes the greatest integer whose square divides $n$.
\end{Theorem}

Note that this dimensions formula becomes even simpler if $\Gamma$ is
normal in $\gamone$ since then the numbers $b_p,f_p$ do not depend on $p$, and
are all equal to, say, $b:=\half[\gamone_\infty:\Gamma_\infty]$,
$f:={{4m} / {(4m,b)}}$. The sums over $p$ in the theorem can then simply be
replaced by $\sharp \Gamma\backslash\Pro$, which equals
${ {[\gamone:\Gamma]} / {2b}}$. In particular, for the group $\Gamma(N)$,
where $b=N$, we find

\begin{Corollary} Let $N,k,m$ be positive integers, $N,k\geq 3.$
Then
\begin{multline*}
\dim S_{k,m}\left(\Gamma(N)\right)=\\
\varphi(N)\psi(N)\left(
mN{\frac {2k-3}  {24}}-
{\frac {d}  8}Q\left({\frac {4m}  {d}}\right)-
{\frac d 4}\sum_\Delta
\left({\frac \Delta {N / {d}}}\right)H(\Delta)
\right).
\end{multline*}
Here $d=(4m,N)$, and $\Delta$ runs through all negative integers dividing
${ {4m} / {d}}$ such that ${ {4m} / {d\Delta}}$ is square-free.
Moreover, $\varphi(N)$ denotes the Euler phi-function, and
$\psi(N)=\#\Pro[\Z/N\Z]=N\prod_{p\vert N}\left(1+{\frac 1 p}\right).$
\end{Corollary}

The simplest instance of this formula occurs for $4m\vert N$ since then the sum
containing the Hurwitz class numbers vanishes. Here, for $k\ge 3$, we obtain
$$
\dim S_{k,m}\left(\Gamma(N)\right)
=m\varphi(N)\psi(N)\left(
N{\frac {2k-3}  {24}}-\half
\right).
$$
This formula was also proved in \cite{K} by considering Jacobi forms as holomorphic sections of certain line bundles, to which the Hirzebruch-Riemann-Roch theorem could be explicitly applied if $4m|N$. 

\section{The general case} 

In this section we compute the dimension formulas
for arbitrary subgroups~$\Gamma$ of $\gamone$. The computations are essentially the same as in
the section~\ref{computation}. However, in view of the various contributions
and cases to consider in the general case, a straightforward calculation would lead
to rather complicated formulas. The main goal of this section is to
state these formulas in a more concise and possibly meaningful way.

To begin with we rewrite
the formula of Theorem~\ref{torsionsfree-gamma-naive}.
To this end we introduce first of all some notation.
As in \cite{S-Z2} we define a function $H_n(\Delta)$ for integers $n\ge1$ and
$\Delta\le 0$. The function $H_1(\Delta)$ equals the Hurwitz class number
$H(\Delta)$, i.e.~$H(0)=-\frac 1{12}$ and
$H(\Delta)$, for $\Delta\not=0$ as recalled in the last but not least
paragraph of section~\ref{computation}.
For general $n\ge 1$
write $(n,\Delta)=a^2b$ with square-free $b$ and set
$$ 
H_n(\Delta)
=
\begin{cases}
a^2b\,\left(\frac{\Delta/a^2b^2}{n/a^2b}\right)\,H_1(\Delta/a^2b^2)
&\text{ if }a^2b^2|\Delta,\\
0 &\text{ otherwise}.
\end{cases}
$$
Furthermore, for integers $k\ge 2$, we define the polynomial\footnote{%
These are, up to a scaling of the argument and a shift in the indices,
the classical Gegenbauer polynomials.
}~$p_k(s)$ as the
coefficient of $x^{k-2}$ in the power series development of
$(1-sx+x^2)^{-1}$. Note that $p_{2k-2}(2)=(2k-3)$ and $p_{2k-2}(0)=(-1)^k$.

Finally, for an exact divisor\footnote{%
i.e.~$n$ and $m/n$ are relatively prime}
$n$ of~$m$ with codivisor $n'=m/n$ and integers
$k\ge 2$, $b\ge 1$ and $t=0,\pm1$ we set\footnote{%
The sum
$s_{k,m,1}^{\text{top}}(n)+s_{k,m,1}^{\text{par.}}(n)
+s_{k,m,-1}^{\text{ell.}}(n)+s_{k,m,0}^{\text{ell.}}(n)
+s_{k,m,1}^{\text{ell.}}(n)$ equals the function~$s_{k,m}(1,n)$
introduced in \cite[Theorem 1]{S-Z2}, which describes the trace of
the Atkin-Lehner operator $W_n$ on the {\it certain space} of modular forms
of level $m$ and weight $2k-2$.
}
\begin{align*}
  s_{k,m;b}^{\text{top}}(n)
  &=
  -          p_{2k-2}(2)  H_{bn'}(0)  
  -\frac 12  Q\left(n'(4n',bn)\right),
  \\
  s_{k,m;b}^{\text{par.}}(n)
  &=
  -\frac12
  (4n,bn')\,
  p_{2k-2}(0)
  \sum_{%
	 \begin{subarray}{c}
		\Delta | 4n/(4n,bn'),\ \Delta < 0\\
		4n/(4n,bn')\Delta\text{ square-free}
	 \end{subarray}
  }
  H_{bn'/(4n,bn')}(\Delta)\\
  s_{k,m;t}^\text{ell.}(n)
  &=
  -\delta\left((t+2)|n\right)\,
  p_{2k-2}\left(\sqrt{t+2}\right)\,
  H_{n'}(t^2-4)
  .
\end{align*}
Here, in the definition of $s_{k,m;b}^{\text{par.}}(n)$,
the sum is over all negative integers~$\Delta$ dividing $4n/(4n,bn')$
such that $4n/(4n,bn')\Delta$ is square-free. Moreover,~$\delta(a|n)$
equals 1 or $0$ accordingly as $a$ divides $n$ or not. Recall
from the previous section that
$Q(n)$, for any
positive integer $n$, denotes the greatest integer whose square divides $n$.

We can now reformulate Theorem~\ref{torsionsfree-gamma-naive}
as follows:

\begin{Theorem}
\label{torsionsfree-gamma}
Let $k$ and $m$ be positive integers, $k\ge2$, and  $\Gamma$ be a
subgroup of finite index in $\gamone$.
Denote by $r$ the number of
cusps of $\Gamma$ and by $b_1$,\dots,~$b_r$ the cusp widths of a complete set
of representatives for the cusps $\Gamma\backslash\Pro$.

If $\Gamma$ is torsion-free and contains no matrices with trace equal to $-2$, then
the dimension of the space of Jacobi cusp forms of weight $k$ and
index $m$ on $\Gamma$ is given by
$$
\dime - \dim\skewJ{3-k}m\Gamma=
\sum_{j=1}^r
\left(
s_{k,m;b_j}^{\text{top}}(1)
+
(-1)^k\,s_{k,m;b_j}^{\text{par.}}(m)
\right)
.
$$
\end{Theorem}

Recall that the cusp width $b_p$ of a cusp $p$ is by definition
equal to $b_p=[\gamone_p:\{\pm1\}\cdot\Gamma_p]$.
To deduce Theorem~\ref{torsionsfree-gamma}
from Theorem~\ref{torsionsfree-gamma-naive} one merely needs to recall that
for any subgroup $\Gamma$ of $\gamone$,
one has
$\sum_{j=1}^r b_r = [\gamone:\{\pm1\}\cdot\Gamma]$.

The formulation of the dimension formula as in Theorem~\ref{torsionsfree-gamma}
has so far no advantage
over the one given in the preceding section. However,
the usage of the auxiliary
functions $s_{k,m}\dots$ will allow us to rewrite more systematically
the dimension formulas for not necessarily torsion-free
groups $\Gamma$, which we shall discuss now.
More precisely, we shall prove the following formula.

\begin{Theorem}
\label{general-gamma-with-minus-one}
Let $k$ and $m$ be positive integers, $k\ge 2$, and  $\Gamma$ be a
subgroup of finite index in $\gamone$.
Denote by $r$ the number of
cusps of $\Gamma$ and by $b_1$,\dots, $b_r$
the cusp widths of a complete set
of representatives for the cusps $\Gamma\backslash\Pro$,
and let $e(0)$ and $e(-1)=e(+1)$ be the number of $\Gamma$-orbits of the
elliptic fixed points of $\Gamma$ which are $\gamone$-equivalent to $i$ and
$e^{2\pi i/3}$, respectively. 

If $\Gamma$ contains the matrix $-1$, then
the dimension of the space of Jacobi cusp forms of weight $k$ and
index $m$ on $\Gamma$ is given by
\begin{align*}
\dime - \dim\skewJ{3-k}m\Gamma
&=
\sum_{j=1}^r
\frac12
\left(
s_{k,m;b_j}^{\text{top}}(1)
+
(-1)^k\,s_{k,m;b_j}^{\text{top.}}(m)
\right)\\
&+
\sum_{j=1}^r
\frac12\left(
s_{k,m;b_j}^{\text{par.}}(1)
+
(-1)^k\,s_{k,m;b_j}^{\text{par.}}(m)
\right)
\\
&+
\sum_{t=-1}^{+1}
\frac{e(t)}2
\left(
s_{k,m;t}^\text{ell.}(1)+(-1)^ks_{k,m;t}^\text{ell.}(m)
\right)
.
\end{align*}
\end{Theorem}

Note that the dimension formula for $S_{k,m}(\Gamma)$, for varying $\Gamma$,
 depends only on
the ``branching scheme'' $b_1$, \dots, $b_r$, $e(0)$, $e(1)$ of $\Gamma$.

\begin{proof}[Proof of Theorem~\ref{general-gamma-with-minus-one}]
In addition to the computation of \S~\ref{computation},
we have first of all to take into account in our general trace formula
the term $I(-1)g(-1)$.
By \cite[Theorem 1, Theorem~2]{S-Z1} this equals
$[\gamone:\Gamma]\,(-1)^{k}\,\frac{2k-3}{24}$. 
In the notation introduced in the beginning of
this section this equals the contribution
of $H_b(0)$ in
$\frac 12 (-1)^k \sum_j s_{k,m;b_j}^{\text{top}}(m)$.
Similarly, the term $I(1)g(1)=m[\gamone:\Gamma](2k-2)/24$ equals
the  contribution
of $H_{bm}(0)$ in
$\frac 12 \sum_j s_{k,m;b_j}^{\text{top}}(1)$. 

Next, let $p$ be a cusp.
Then there is a positive integer $b$ such that
the groups $\Gamma_p$ and $\Gamma_p\cap\gam$ are
$\gamone$-conjugate to
$\langle\pm 1\rangle \times \langle \mat {1}{b}{0}{1}\rangle$
and $\langle\mat {1}{bf}{0}{1}\rangle$ with $f=4m/(4m,b)$, respectively.
Accordingly, we find
$t_p=t_p^+ + t_p^-$, where
$$
t_p^\varepsilon
=\sum_{0<\nu\leq f}
I\left(\varepsilon \mat {1}{b\nu}{0}{1}\right)\,
g\left(\varepsilon \mat {1}{b\nu}{0}{1}\right)
.
$$
Here $t_p^+$ equals one half of the $t_p$ of section~\S~\ref{computation}
(note that in the case considered here $[\Gamma_p:\Gamma_p\cap\Gamma(4m)]=2f$ due to the presence
of $-1$ in $\Gamma$).
Accordingly, $t_p^+$ equals $\frac12(-1)^k\sum_j s_{k,m;n_j}^{\text{par.}}(m)$
plus the contribution of the $Q$-terms in
$\frac12 \sum_j s_{k,m;n_j}^{\text{top}}(1)$.

For the calculation of $t_p^-$, we use
\begin{gather*}
I\left(-(1,b\nu;0,1)\right)
=-\frac1{4f}\,i^{1-2k}\left(1-i\,C(\nu/f)\right),\\
g\left(-(1,b\nu;0,1)\right)
=-i\sum_{\lambda\bmod 2}\e{\frac{mb\nu\lambda^2}4}
\end{gather*}
(cf.~\cite[Theorem~1,~2]{S-Z1}).
By a similar calculation as in section~\S~\ref{computation} we find
$$
t_p^-
=
-\frac {(-1)^k}4 \left(Q((4,bm))+\left(\frac{-4}{bm}\right)H(-4)\right)
.
$$
But this equals the $Q$-term in $\frac 12 (-1)^k s_{k,m;b}^{\text{top}}(m)$
plus $\frac12 s_{k,m;b}^{\text{par.}}(1)$.

If $A=(a,b;c,d)$ is an elliptic matrix in $\Gamma$ with trace $t$,
then by \cite[Theorem~1,~2]{S-Z1},
we have
$$
I(A)=\frac1{|\Gamma_e|}\sym{sign}(c)\,
\frac{\rho^{3/2-k}}{\rho-\overline{\rho}},
\quad
g(A)=-i|t-2|^{-3/2}
\sum_{\lambda,\mu\bmod t-2}\e{\frac m{t-2}Q_A(\lambda,\mu)}
.
$$
Here $\rho$ and $\overline\rho$ are the roots of $x^2-tx+1=0$
such that the imaginary part of $\rho$ and $c$ have the same sign,
$\Gamma_e$ is the stabilizer in $\Gamma$ of the elliptic fixed point
$e$ of $A$ in the upper half plane, and
$Q_A(\lambda,\mu)=b\lambda^2+(d-a)\lambda\mu-c\mu^2$.
Note that
$I(A)=-\overline {I(A^{-1})}$ and that the same identity holds true for $g(A)$.
Thus, $A$ and $A^{-1}$ add the contribution
$$
t_A
=
2\sym{Re}(I(A))\sym{Re}(g(A))-2\sym{Im}(I(A))\sym{Im}(g(A))
$$
to our general trace formula.
One easily verifies
$$
\sym{Re}(I(A))
=
-\frac {p_{2k-2}\left(\sqrt{2+t}\right)}{2|\Gamma_e|\sqrt{2+t}},
\quad
\sym{Im}(I(A))
=
-(-1)^k\frac {p_{2k-2}\left(\sqrt{2-t}\right)}{2|\Gamma_e|\sqrt{2-t}}
$$
(using $\rho=\left(\frac{\sqrt{t+2}+\sqrt{t-2}}2\right)^2$)
and
\begin{align*}
\sym{Re}(g(A))
=
\frac12\delta(2+t=1)\,|\Gamma_e|\,\sqrt{2+t}\,H_m(t^2-4),
\\
\sym{Im}(g(A))
=
-\frac12\delta((2-t)|m)\,|\Gamma_e|\,\sqrt{2-t}\,H_1(t^2-4)
\end{align*}
(by a case by case inspection; note that $I(A)$ and $g(A)$
depend only on the
$\gamone$-conjugacy class of~$A$, thus it suffices to verify
the latter two formulas for $A=\mat{0}{-1}{1}{0}$,
$A=\mat{0}{-1}{1}{1}$ and $A=\mat{-1}{-1}{1}{0}$, respectively).
Hence
$$
t_A=
\frac12\left(
s_{k,m;t}^{\text{ell.}}(1)
+(-1)^k
s_{k,m;-t}^{\text{ell.}}(m)
\right)
.
$$
It is now clear that the contributions of the elliptic matrices
add up to the
term as stated in the theorem.
\end{proof}

We leave it to the reader to verify the last theorem, which describes the
remaining case, i.e.~the case of a $\Gamma$ which does not contain the matrix $-1$
but possibly elliptic fixed points and irregular cusps
(i.e.~cusps $p$ such that $\Gamma_p$ is generated by an element with
negative trace). Here the corresponding dimension formulas run as follows: 
\begin{Theorem}
\label{gamma-without-minus-one}
Let the notations be as in Theorem~\ref{general-gamma-with-minus-one}.
Suppose that $\Gamma$ does not contain the matrix $-1$, and let
$b_1$, \dots,$b_{r_1}$ the cups widths of the regular cusps
and $b_{r_1+1}$, \dots,$b_{r}$ the cups widths of the irregular ones.
Then one has
\begin{align*}
\dime - \dim\skewJ{3-k}m\Gamma
&=
\sum_{j=1}^{r_1}
\left(s_{k,m;b_j}^{\text{top}}(1)+(-1)^ks_{k,m;b_j}^{\text{par.}}(m)\right)
\\
&+\sum_{j=r_1+1}^{r_2}
\frac 12
\left(
s_{k,m;2b_j}^{\text{top}}(1)
+
(-1)^k
s_{k,m;2b_j}^{\text{par.}}(m)
\right)
\\
&+\sum_{j=r_1+1}^{r_2}
\left(
s_{k,m;b_j}^{\text{par}}(1)
+
(-1)^k
s_{k,m;b_j}^{\text{top}}(m)
\right)
\\
&-\sum_{j=r_1+1}^{r_2}
\frac 12
\left(
s_{k,m;2b_j}^{\text{par}}(1)
+
(-1)^k
s_{k,m;2b_j}^{\text{top}}(m)
\right)
\\
&+e(-1)
\left(
s_{k,m;-1}^{\text{ell.}}(1)+(-1)^ks_{k,m;+1}^{\text{ell.}}(m)
\right)
.
\end{align*}
\end{Theorem}

\section{Concluding remarks} 

Theorem~\ref{torsionsfree-gamma-naive} to Theorem~\ref{gamma-without-minus-one}
summarize the dimension formulas for holomorphic Jacobi cusp forms
of arbitrary integral weight $k\ge 2$ and integral index $m\ge1$
on arbitrary subgroups $\Gamma$ of $\gamone$. However,
for the important case $k=2$, we would still have to compute the
term $\dim\skewJ{3-k}m\Gamma$ to obtain an explicit
formula. In principle this computation could be done, however,
this seems to be a rather cumbersome task. In essence, this computation
would reduce to an analysis of the action of $\gamone$
on the space of modular form of weight~$\frac12$. For an example of
this kind of computation the reader is referred to \cite{I-S},
where we proved vanishing results for spaces of (holomorphic) Jacobi
forms of weight~1 on groups $\Gamma_0(l)$.  

The general trace formula of~\cite{S-Z1} admits also to derive explicit
dimension formulas for spaces of Jacobi forms with characters,
like e.g.~for the spaces $S_{k,m}(\Gamma_0(l),\chi)$, where~$\chi$
is a Dirichlet character modulo~$l$. It also admits the derivation of
explicit formulas for the traces of Atkin-Lehner operators $W_n$
(as considered in~\cite{S-Z2} for Jacobi forms on $\gamone$) on spaces of Jacobi forms on general $\Gamma$.
It it the very likely
that the function $s^*_{k,m;b}(n)$
for nontrivial divisors $n$ of $m$ are related to these trace formulas.

It might be interesting to ask for the geometric interpretation
of the decomposition of
the dimension formulas into the $s_{k,m}\dots$-parts.
A clue to this would be the article~\cite{K}.

Finally, it might be interesting to compare the dimension formulas
for Jacobi forms to the dimension formulas for ordinary elliptic
modular forms. For example, the dimension of the space $S_{2k-2}(m)$
of modular cups forms
of weight $2k-2$ on $\Gamma_0(m)$ is given by
$$
\dim S_{2k-2}(m)
=
\sum_{\begin{subarray}{c}m'|m\\ \frac m{m'}\text{ square-free}\end{subarray}}
\left(
s_{k,m',1}^{\text{top}}(1)
+
s_{k,m',1}^{\text{par}}(1)
+
\sum_{t=-1}^{+1}s_{k,m',t}^{\text{ell}}(1)
\right)
$$
(cf.~\cite{S-Z2}). This reflects the existence of a certain
natural subspace of $S_{2k-2}(m)$, whose dimension equals
the term corresponding to $m$, and which, in the cited article, was proved to be Hecke-equivariantly
isomorphic to $S_{k,m}(\gamone)$.
Similar lifting maps exist also for Jacobi forms on proper subgroups
of $\gamone$\footnote{%
However, to our knowledge this has never been worked out in detail
for groups different from $\gamone$.},
 and a comparison of dimension formulas might give a first
clue towards an explicit description of the images of such liftings.
These liftings suggest Hecke-equivariant relations e.g.~between
Jacobi forms of index~1 on $\Gamma_0(l)$
and Jacobi forms of index $l$ and on $\Gamma(1)$. Again, our dimension formulas
may help to pinpoint what exactly one should expect. From our formulas
we find e.g., for primes $p\equiv 1\bmod 12$ and even $k\ge4$,
that the dimension of
$\dim S_{k,1}(\Gamma_0(p))$ equals the dimension of
$S_{k,1}(\gamone) \oplus S_{k,p}(\gamone) \oplus S_{k,p}^+(\gamone)$
(assuming the so far unproved fact\footnote{%
We hope to prove this eventually
in another article.
}
that the dimension of the space of skew-holomorphic cusp forms $S_{k,p}^+(\gamone)$
is given by the same formula as for $S_{k,p}(\gamone)$, but with the $(-1)^k$
replaced by $-(-1)^k$.).


\vskip3em
\begin{flushleft}
Nils-Peter Skoruppa\\
Universit\"at Siegen --- Fachbereich Mathematik\\
Walter-Flex-Stra{\ss}e 3, D-57068 Siegen, Germany\\
www.countnumber.de
\end{flushleft}

\end{document}